\documentclass[openacc]{rsproca_new}

\usepackage{amssymb,amsmath,amsthm}
\usepackage{tikz}
\usetikzlibrary{cd}

\newtheorem{thm}{\bf Theorem}[section]
\newtheorem{coro}[thm]{\bf Corollary}
\newtheorem{lem}[thm]{\bf Lemma}
\newtheorem{prop}[thm]{\bf Proposition}

\newtheorem{athm}{\bf Theorem}

\theoremstyle{definition}
\newtheorem{defn}[thm]{\bf Definition}
\newtheorem{notation}[thm]{\bf Notation}
\newtheorem{terminology}[thm]{\bf Terminology}
\newtheorem{eg}[thm]{\bf Example}
\newtheorem{construction}[thm]{\bf Construction}

\theoremstyle{remark}
\newtheorem{rmk}[thm]{\bf Remark}

\newcommand{\longhookrightarrow}{\ensuremath{\lhook\joinrel\longrightarrow}}
\newcommand{\rmod}[1]{\ensuremath{#1\text{-}\mathrm{Mod}}}
\newcommand{\mbar}{\ensuremath{\,\overline{\! M}}}
\newcommand{\mhat}{\ensuremath{\,\widehat{\! M}}}

\newcommand{\cB}{\mathcal{B}}

\newcommand{\cI}{\mathcal{I}}

\newcommand{\cM}{\mathcal{M}}
\newcommand{\cN}{\mathcal{N}}

\newcommand{\cT}{\mathcal{T}}

\newcommand{\cX}{\mathcal{X}}

\newcommand{\bC}{\mathbb{C}}

\newcommand{\bQ}{\mathbb{Q}}
\newcommand{\bR}{\mathbb{R}}

\newcommand{\bZ}{\mathbb{Z}}

\titlehead{Research}

\begin{document}
\title{Homology stability for asymptotic monopole moduli spaces}
\author{Martin Palmer$^{1}$ and Ulrike Tillmann$^{2}$}
\address{$^{1}$Institutul de Matematică Simion Stoilow al Academiei Române, 21 Calea Griviței, Bucharest, Romania \\ $^{1}$ORCID: 0000-0002-1449-5767 \\ $^{2}$Mathematical Institute, University of Oxford, Andrew Wiles Building, Oxford, OX2 6GG, UK \\ $^{2}$Isaac Newton Institute, University of Cambridge, Cambridge, CB3 0EH, UK \\ $^{2}$ORCID: 0000-0002-8076-7660}
\subject{58D27, 55R80, 53C07, 81T13}
\corres{Ulrike Tillmann \\ \email{tillmann@maths.ox.ac.uk}}

\begin{abstract}
We prove homological stability for two different flavours of asymptotic monopole moduli spaces, namely moduli spaces of \emph{framed Dirac monopoles} and moduli spaces of \emph{ideal monopoles}. The former are Gibbons-Manton torus bundles over configuration spaces whereas the latter are obtained from them by replacing each circle factor of the fibre with a monopole moduli space by the Borel construction. They form boundary hypersurfaces in a partial compactification of the classical monopole moduli spaces. Our results follow from a general homological stability result for configuration spaces equipped with non-local data.
\end{abstract}

\begin{fmtext}

\section*{Introduction}

The topology of the moduli spaces of magnetic monopoles $\cM_k$ has been the subject of intensive study for many decades. By a theorem of Donaldson~\cite{Donaldson1984Nahmsequationsclassification}, they have a model as spaces of rational functions on $\bC P^1$. Via this model, their homotopy and homology groups are known to stabilise as $k\to\infty$ by a theorem of Segal~\cite{Segal1979topologyofspaces} and their homology (both stable and unstable) was completely computed by \cite{CohenCohenMannMilgram1991topologyrationalfunctions} in terms of the homology of the braid groups, which is completely known~\cite{CohenLadaMay1976}.

The moduli spaces $\cM_k$ are non-compact manifolds. Recently, a partial compactification of $\cM_k$ has been constructed by Kottke and Singer~\cite{KottkeSinger2022} by adding certain boundary hypersurfaces $\cI_\lambda$ to $\cM_k$ indexed by partitions $\lambda = (k_1,\ldots,k_r)$ of $k$.

\end{fmtext}

\maketitle

Points in these boundary hypersurfaces are thought of as ``ideal'' monopoles of total charge $k$, with $r$ ``clusters'' centred at different points in $\bR^3$, with charges $k_1,\ldots,k_r$, which are ``widely separated'' but nevertheless interact.

Our main theorem proves a homology stability result for these ideal monopole moduli spaces as the number of clusters of a fixed charge $c\geq 1$ goes to infinity:

\begin{athm}
\label{athm:ideal-monopoles}
Fix a positive integer $c$ and a tuple $\lambda = (k_1,\ldots,k_r)$, of fixed length $r$, of positive integers $k_i \neq c$. Write $\lambda[n]_c = (k_1,\ldots,k_r,c,\ldots,c)$, where $c$ appears $n$ times. There are natural stabilisation maps
\begin{equation}
\label{eq:stab-maps-ideal-monopoles}
\cI_{\lambda[n]_c} \longrightarrow \cI_{\lambda[n+1]_c}
\end{equation}
that induce isomorphisms on homology in all degrees $\leq n/2 - 1$ with $\bZ$ coefficients and in all degrees $\leq n/2$ with field coefficients.
\end{athm}

We also prove an analogous result for \emph{moduli spaces of framed Dirac monopoles} (in other words \emph{Gibbons-Manton torus bundles}; see \S\ref{ss:boundary-hypersurfaces} for the definitions) and, more generally, Gibbons-Manton $\mathbf{Z}$-bundles for any sequence $\mathbf{Z}$ of path-connected $S^1$-spaces; see Theorems \ref{thm:stability} and \ref{thm:GM-Zstar-stability}.

These results follow from a general homology stability result (Proposition \ref{prop:hs-with-non-local-data}) for unordered configuration spaces with \emph{non-local parameters}. Homology stability for configuration spaces whose points are labelled by elements of a fixed space $X$ is well-known; these are configuration spaces with \emph{local} parameters. However, the ideal monopole moduli spaces $\cI_\lambda$ are \emph{non-local}. The key observation in \S\ref{s:stability} is that homology stability only requires the parameters associated to a configuration to satisfy much weaker properties, which allows us to consider interesting non-local parameters. In \cite{PalmerTillmann2021}, we recently proved a different homology stability result for non-local configuration spaces, namely for \emph{configuration-section spaces}; this encouraged us to try to prove homology stability also in the context of ideal monopole moduli spaces. Proposition \ref{prop:hs-with-non-local-data} is the abstract general result that applies in our situation in the present paper. Though similar in nature, it neither is implied by nor implies the homology stability result in \cite{PalmerTillmann2021}.

\paragraph{Outline.}

We first recall some background on moduli spaces of magnetic monopoles in \S\ref{s:bg}: first on the moduli spaces themselves in \S\ref{ss:monopole-moduli-space} and then on their partial compactifications introduced by \cite{KottkeSinger2022} in \S\ref{ss:boundary-hypersurfaces}, whose boundary hypersurfaces are the ideal monopole moduli spaces. In \S\ref{s:stability} we then prove a general homology stability result for configuration spaces equipped with ``non-local'' data, deducing it from twisted homological stability for configuration spaces \cite{Palmer2018Twistedhomologicalstability} (see also \cite{Krannich2019Homologicalstabilitytopological}). In \S\ref{s:stability-monopoles} we apply it to prove our main theorem, homology stability for ideal monopole moduli spaces, as well as an extension (Theorem \ref{thm:GM-Zstar-stability}) to Gibbons-Manton $\mathbf{Z}$-bundles more generally.

\section{Monopole moduli space and boundary hypersurfaces}
\label{s:bg}

\subsection[Monopole moduli space]{Monopole moduli space.}
\label{ss:monopole-moduli-space}

We briefly recall from \cite{AtiyahHitchin1988geometrydynamicsmagnetic} some different monopole moduli spaces and the relations between them.

A \emph{magnetic monopole} on $\bR^3$ is a pair consisting of a connection $A$ on the trivial principal $SU(2)$-bundle on $\bR^3$ together with a field $\phi$ taking values in the associated Lie algebra $\mathfrak{su}(2)$. Fixing a framing, these may be viewed, respectively, as a smooth $1$-form and a smooth function on $\bR^3$ taking values in $\mathfrak{su}(2)$, which we may identify topologically as $\mathfrak{su}(2) \cong \bR^3$. These data $A$ and $\phi$ must satisfy the \emph{Bogomolny equations} and a certain finiteness condition; see \cite[pp.~14--15]{AtiyahHitchin1988geometrydynamicsmagnetic} for details. This finiteness condition implies that $\phi(x) \neq 0$ for $\lvert x \rvert$ sufficiently large, so the restriction of $\phi$ to $\bR^3 \smallsetminus B_R(0)$ takes values in $\mathfrak{su}(2) \smallsetminus \{0\}$ for $R \gg 0$. The degree of this map is the charge of the monopole, and is always positive. The set of all magnetic monopoles of charge $k\geq 1$, up to gauge equivalence (automorphisms of the trivial bundle $\bR^3 \times \mathfrak{su}(2) \to \bR^3$), suitably topologised, is the \emph{monopole moduli space} $\cN_k$. A slight variation of the construction, quotienting by a smaller gauge group, yields a different space $\cM_k$ related to $\cN_k$ by a principal $S^1$-bundle
\begin{equation}
\label{eq:Mk-to-Nk}
\cM_k \longrightarrow \cN_k = \cM_k / S^1.
\end{equation}
Translation of solutions to the Bogomolny equations in $\bR^3$ also defines a principal $\bR^3$-bundle
\begin{equation}
\label{eq:Nk-to-Nk0}
\cN_k \longrightarrow \cM_k^0 = \cN_k / \bR^3.
\end{equation}

The spaces $\cM_k$ and $\cM_k^0$ admit the structure of hyperKähler manifolds of dimensions $4k$ and $4k-4$ respectively. For charge $k=1$ we have $\cM_1^0 = pt$ (and $\cM_1 \cong S^1 \times \bR^3$) and for $k=2$, the $4$-manifold $\cM_2^0$ is known as the \emph{Atiyah-Hitchin manifold} and has been studied in detail in \cite{AtiyahHitchin1988geometrydynamicsmagnetic}.

By \cite{Donaldson1984Nahmsequationsclassification}, $\cM_k$ is homeomorphic to the space $R_k$ of degree-$k$ rational self-maps of $\bC P^1$ that send $\infty$ to $0$. Thus, it is also homeomorphic to the space $R'_k$ of degree-$k$ rational self-maps of $\bC P^1$ that send $\infty$ to $1$. The points of the space $R'_k$ may conveniently be described as pairs $(p,q)$ of coprime monic polynomials with coefficients in $\bC$, both of degree $k$. Identifying these polynomials with their sets of roots, we obtain a natural embedding
\[
R'_k \longhookrightarrow SP^k(\bC) \times SP^k(\bC)
\]
whose image consists of all pairs $(A,B)$ of multi-subsets of $\bC$ that are disjoint. On the other hand, the space $R_k$ is convenient in that the circle action is easy to see: under the isomorphism $\cM_k \cong R_k$, the circle action is given simply by multiplying rational self-maps of $\bC P^1$ by $e^{i \theta}$.

The fundamental group of $\cM_k$ is $\bZ$, by \cite[Proposition~6.4]{Segal1979topologyofspaces}. Also, by \cite[chapter~2]{AtiyahHitchin1988geometrydynamicsmagnetic}, the fundamental group of $\cN_k$ is $\bZ/k$ and the projection map \eqref{eq:Mk-to-Nk} induces the reduction-mod-$k$ map $\bZ \twoheadrightarrow \bZ/k$. It follows from the long exact sequence that \eqref{eq:Mk-to-Nk} induces isomorphisms on $\pi_i$ for all $i\geq 2$, so $\cM_k$ and $\cN_k$ have the same universal cover, up to homotopy equivalence, which is denoted by $\cX_k$.

There are stabilisation maps $\cM_k \to \cM_{k+1}$, which may be defined under the isomorphism $\cM_k \cong R_k$ by adding to a given rational self-map a new zero and a new pole ``far away'' from the origin. (This is not invariant under the circle action, so it does not descend to a stabilisation map on the moduli spaces $\cN_k$.) The stabilisation maps $\cM_k \to \cM_{k+1}$ induce isomorphisms on homotopy groups (and hence also homology groups) in a stable range, by \cite{Segal1979topologyofspaces}. Lifting to universal covers, it follows that there are also stabilisation maps $\cX_k \to \cX_{k+1}$ that induce isomorphisms on homotopy (and homology) groups in a stable range.

By the main theorem of \cite{Segal1979topologyofspaces}, the homotopy colimit of the stabilisation maps $\cM_k \to \cM_{k+1} \to \cdots$ is weakly equivalent to $\Omega_0^2 S^2$. Thus the stable homology of $\cM_k$ is the homology of $\Omega_0^2 S^2$ and the stable homology of $\cX_k$ is the homology of the universal cover of $\Omega_0^2 S^2$. Moreover, the \emph{unstable} homology of $\cM_k$ (i.e.~its homology outside of the stable range) is also known: by the main result of \cite{CohenCohenMannMilgram1991topologyrationalfunctions, CohenCohenMannMilgram1993homotopytyperational}, the homology of $\cM_k$ is isomorphic to the group homology of the braid group $B_{2k}$, which is completely computed \cite{CohenLadaMay1976}. The \emph{rational} unstable homology of $\cX_k$ has also been computed by \cite{SegalSelby1996}, and is significantly more complicated than the rational unstable homology of $\cM_k$: the rational homology $H_*(\cM_k;\bQ)$ is the same as that of the circle, so it has total dimension $2$, whereas \cite{SegalSelby1996} shows that the rational homology $H_*(\cX_k;\bQ)$ has total dimension $k$, concentrated in degrees of the form $2(k-d)$ where $d$ is a divisor of $k$.

\begin{notation}
\label{notation:Mkc}
The principal bundles \eqref{eq:Mk-to-Nk} and \eqref{eq:Nk-to-Nk0} arise from a principal (in particular free) action of the product $S^1 \times \bR^3$ on on $\cM_k$. If we first quotient by $\bR^3$ (Euclidean translations) we obtain a principal $\bR^3$-bundle
\begin{equation}
\label{eq:Mk-to-Mkc}
\cM_k \longrightarrow \cM_k^c = \cM_k / \bR^3 .
\end{equation}
In particular, we have a homotopy equivalence $\cM_k^c \simeq \cM_k$. (The superscript $c$ stands for \emph{centred} monopoles.) The quotient $\cM_k^c$ is a $(4k-3)$-dimensional manifold and there is a principal $S^1$-bundle
\begin{equation}
\label{eq:Mkc-to-Mk0}
\cM_k^c \longrightarrow \cM_k^0 = \cM_k^c / S^1 = \cN_k / \bR^3 .
\end{equation}
\end{notation}

\subsection[Boundary hypersurfaces]{Boundary hypersurfaces.}
\label{ss:boundary-hypersurfaces}

Kottke and Singer \cite{KottkeSinger2022} have constructed a partial compactification of $\cM_k^c \simeq \cM_k$ of the form
\begin{equation}
\label{eq:partial-compactification}
\overline{\cM}_k^c = \bigsqcup_\lambda \cI^c_\lambda
\end{equation}
with strata indexed by sequences $\lambda = (k_1,\ldots,k_r)$ of positive integers that sum to $k$. The stratum $\cI^c_{(k)}$ is the interior $\cM^c_k$ of $\overline{\cM}^c_k$ and the union of all strata $\cI^c_\lambda$ for $\lambda \neq (k)$ is the boundary of $\overline{\cM}^c_k$. Points in $\cI^c_\lambda$ are called \emph{centred ideal} monopoles associated to the partition $\lambda$.

We will not recall here the construction of $\cI^c_\lambda$ in \cite{KottkeSinger2022}; instead we will take an alternative characterisation of $\cI^c_\lambda$ to be its definition (see Definitions \ref{def:ideal} and \ref{def:Dirac} and Remark \ref{rmk:definitions}). To begin with, we recall the definitions of ordered and unordered configuration spaces.

\begin{defn}
\label{def:configuration-spaces}
For any space $M$, let us write $F_r(M) = \{ (v_1,\ldots,v_r) \in M^r \mid v_i \neq v_j \text{ for } i\neq j \}$ for the \emph{ordered configuration space} of $r$ points in $M$, topologised as a subspace of the product $M^r$. We also write $C_r(M) = F_r(M) / \Sigma_r$ for the \emph{unordered configuration space} of $r$ points in $M$.
\end{defn}

Recall (see for example \cite[Theorem V.1.1]{FadellHusseini2001}) that the degree-$(d-1)$ cohomology of $F_r(\bR^d)$ is given by:
\begin{equation}
\label{eq:cohomology-basis}
H^{d-1}(F_r(\bR^d);\bZ) \;\cong\; \bZ \bigl\lbrace \alpha_{ij} \mid 1\leq i < j \leq r \bigr\rbrace ,
\end{equation}
where $\alpha_{ij}$ is the pullback of a generator of $H^{d-1}(S^{d-1};\bZ)$ along the map $\iota_{ij} \colon F_r(\bR^d) \to S^{d-1}$ given by the formula
\[
\mathbf{x} = (x_1, \dots, x_r) \longmapsto \frac{x_i-x_j}{\lvert x_i - x_j \rvert}.
\]

Since principal $S^1$-bundles over a space $X$ are classified by $H^2(X;\bZ)$, this means that principal $S^1$-bundles over $F_r(\bR^3)$ are classified by integer linear combinations of the $\alpha_{ij}$.
(One dimension lower, the same data classifies principal $\bZ$-bundles over $F_r(\bR^2)$, in other words regular coverings of $F_r(\bR^2)$ with infinite cyclic deck transformation group.)

\begin{defn}[{\cite[Definition 4.6 and the paragraph preceding it]{KottkeSinger2022}}]
\label{def:GM-bundles}
For a sequence of integers $\lambda = (k_1,\ldots,k_r)$, the corresponding \emph{Gibbons-Manton circle factors} are the principal $S^1$-bundles
\[
S_{\lambda,j} \longrightarrow F_r(\bR^3),
\]
for $j \in \{1,\ldots,r\}$, corresponding to the element $\sum_{i \in \{1,\ldots,r\}, i \neq j} k_i.\alpha_{ij}$, where we define $\alpha_{ij} = -\alpha_{ji}$ if $i>j$. The \emph{Gibbons-Manton torus bundle} weighted by $\lambda$ is the principal $T^r$-bundle
\begin{equation}
\label{eq:GM-bundle}
\widetilde{\cT}_\lambda = \bigoplus_{j=1}^r S_{\lambda,j} \longrightarrow F_r(\bR^3).
\end{equation}
A point in $S_{\lambda,j}$ may be thought of as an ordered configuration together with a non-local circle parameter encoding the interaction of the $j$th particle with all other particles, weighted by $\lambda$. A point in $\widetilde{\cT}_\lambda$ may similarly be thought of as an ordered configuration together with $r$ non-local circle parameters, each encoding the interaction of one of the particles with all of the others (again, weighted by $\lambda$).
\end{defn}

\begin{defn}
\label{def:Gibbons-Manton-config-space}
The symmetric group $\Sigma_r$ acts on $F_r(\bR^3)$ by permuting the particles. Let $\Sigma_\lambda \leq \Sigma_r$ be the stabiliser of $\lambda = (k_1,\ldots,k_r) \in \bZ^r$ under the obvious permutation action of $\Sigma_r$ on $\bZ^r$. Then the action of $\Sigma_\lambda$ on $F_r(\bR^3)$ lifts to a well-defined action on $\widetilde{\cT}_\lambda$. The \emph{Gibbons-Manton configuration space} is the quotient space $\cT_\lambda = \widetilde{\cT}_\lambda / \Sigma_\lambda$. Note that there is a principal $T^r$-bundle
\begin{equation}
\label{eq:GM-bundle-2}
\cT_\lambda \longrightarrow F_r(\bR^3) / \Sigma_\lambda .
\end{equation}
In particular, when $k_1 = k_2 = \cdots = k_r$, we have $\Sigma_\lambda = \Sigma_r$ and $\cT_\lambda$ is a principal $T^r$-bundle over the unordered configuration space $C_r(\bR^3)$.
\end{defn}

\begin{rmk}
One may make analogous definitions for Euclidean spaces $\bR^d$ in general, replacing $S^1 = K(\bZ,1)$ with $K(\bZ,d-2)$, so that $\cT_\lambda$ is a principal $K(\bZ,d-2)^r$-bundle over $F_r(\bR^d)$. For example, when $d=2$, it is a regular covering space with deck transformation group isomorphic to $\bZ^r$. In particular, for $d=2$ and $\lambda = (1,1,\ldots,1)$, it is the regular covering space corresponding to the homomorphism
\[
\varphi_r \colon \pi_1(F_r(\bR^2)) = PB_r \longrightarrow \bZ^r
\]
that records, for each $1\leq i\leq r$, the total winding number of the $i$th strand of a given pure braid around the other $r-1$ strands. This is a disconnected covering with components indexed by $\mathrm{coker}(\varphi_r)$; each connected component is a classifying space for the subgroup $\mathrm{ker}(\varphi_r) \leq PB_r$ consisting of those pure braids $b$ where each strand of $b$ has zero total winding number around the other $r-1$ strands:
\[
\bigsqcup_{\mathrm{coker}(\varphi_r)} B(\mathrm{ker}(\varphi_r)) \longrightarrow F_r(\bR^2).
\]
\end{rmk}

\begin{defn}
\label{def:ideal}
The \emph{moduli space of ideal monopoles} of weight $\lambda$ is defined as follows. Recall that the monopole moduli space $\cM_k$ is equipped with a circle action. The product $\cM_{k_1} \times \cdots \times \cM_{k_r}$ is therefore equipped with an action of the torus $T^r$. We define $\widetilde{\cI}_\lambda$ to be the total space of the fibre bundle associated to the principal $T^r$-bundle $\widetilde{\cT}_\lambda$ by changing the fibre to $\cM_{k_1} \times \cdots \times \cM_{k_r}$. In other words, it is the Borel construction
\[
\widetilde{\cI}_\lambda = \widetilde{\cT}_\lambda \times_{T^r} \bigl( \cM_{k_1} \times \cdots \times \cM_{k_r} \bigr) \longrightarrow F_r(\bR^3).
\]
We then define $\cI_\lambda = \widetilde{\cI}_\lambda / \Sigma_\lambda$, where $\Sigma_\lambda$ acts diagonally on $\widetilde{\cT}_\lambda$ (see Definition \ref{def:Gibbons-Manton-config-space}) and on the product $\cM_{k_1} \times \cdots \times \cM_{k_r}$. The \emph{moduli space of ideal monopoles} of weight $\lambda$ is this space $\cI_\lambda$. It is the total space of a fibre bundle
\begin{equation}
\label{eq:ideal-monopole-bundle}
\pi \colon \cI_\lambda \longrightarrow F_r(\bR^3) / \Sigma_\lambda
\end{equation}
with fibre $\cM_{k_1} \times \cdots \times \cM_{k_r}$.
\end{defn}

\begin{rmk}
This is not yet the boundary stratum $\cI^c_\lambda$ constructed by \cite{KottkeSinger2022} in their partial compactification of $\cM_k^c$, since it has the wrong dimension. Recall that the dimension of $\cM_k^c$ is $4k-3$, so its boundary strata must have dimension $4k-4$, whereas the dimension of $\cI_\lambda$ is $4k + 3r$. The definition of $\cI^c_\lambda$ is similar to that of $\cI_\lambda$ (and these two spaces are homotopy equivalent; see Remark \ref{rmk:I-and-Ic}), using the centred moduli spaces $\cM^c_{k_i}$ instead of $\cM_{k_i}$ and using a centred version of the configuration space, which we define next.
\end{rmk}

\begin{defn}
\label{defn:ordered-centred}
The ordered centred configuration space $F_r^c(\bR^3) \subseteq F_r(\bR^3)$ is defined to be the space of all ordered configurations $(x_1,\ldots,x_r)$ in $F_r(\bR^3)$ such that
\begin{equation}
\label{eq:centred-config-space-condition}
\sum_{i=1}^r x_i = 0 \qquad\text{and}\qquad \sum_{i=1}^r \lvert x_i \rvert^2 = 1
\end{equation}
and has dimension $3r-4$. The unordered version $C_r^c(\bR^3) \subseteq C_r(\bR^3)$ is defined similarly and we have $C_r^c(\bR^3) = F_r^c(\bR^3) / \Sigma_r$.
\end{defn}

\begin{defn}
\label{def:ideal-centred}
The \emph{moduli space of centred ideal monopoles} of weight $\lambda$ is defined as follows. Analogously to Definition \ref{def:ideal}, consider the Borel construction
\[
\widetilde{\cI}^c_\lambda = \widetilde{\cT}^c_\lambda \times_{T^r} \bigl( \cM^c_{k_1} \times \cdots \times \cM^c_{k_r} \bigr) \longrightarrow F^c_r(\bR^3),
\]
where $\widetilde{\cT}^c_\lambda$ is the restriction of $\widetilde{\cT}_\lambda \to F_r(\bR^3)$ to $F_r^c(\bR^3) \subseteq F_r(\bR^3)$. We then define $\cI^c_\lambda = \widetilde{\cI}^c_\lambda / \Sigma_\lambda$, which is the total space of a fibre bundle
\begin{equation}
\label{eq:ideal-monopole-bundle-centred}
\pi \colon \cI^c_\lambda \longrightarrow F^c_r(\bR^3) / \Sigma_\lambda
\end{equation}
with fibre $\cM^c_{k_1} \times \cdots \times \cM^c_{k_r}$.
\end{defn}

\begin{rmk}
\label{rmk:I-and-Ic}
Since the inclusion $F_r^c(\bR^3) \subseteq F_r(\bR^3)$ and the projection \eqref{eq:Mk-to-Mkc} are homotopy equivalences, we also have
\[
\cI^c_\lambda \simeq \cI_\lambda .
\]
They are therefore interchangeable when studying their homotopical properties individually. However, they are not homeomorphic, and $\cI^c_\lambda$ (rather than $\cI_\lambda$) is the boundary stratum corresponding to $\lambda$ in the partial compactification of \cite{KottkeSinger2022}. Note that the space $\cI^c_\lambda$ has the correct dimension, namely $(3r-4) + \sum_{i=1}^r (4k_i - 3) = 3r-4+4k-3r = 4k-4$.

However, since we focus in this paper on the homological properties of $\cI_\lambda$, the difference between $\cI_\lambda$ and $\cI^c_\lambda$ will not be relevant to us.
\end{rmk}

\begin{terminology}
\label{widely-separated}
When $\lambda = (1,1,\ldots,1)$, the moduli space $\cI_\lambda$ is called the \emph{moduli space of widely separated magnetic monopoles}. This terminology follows the intuition that points $x \in \cI_\lambda$ should be thought of as monopoles of total charge $k$, with $r$ different ``clusters'' centred at the points $\pi(x)$, with charges $k_i$, which are ``widely separated'' but nevertheless interact: these interactions are encoded in the structure group $T^r$ of the bundle \eqref{eq:ideal-monopole-bundle}.
\end{terminology}

\begin{defn}
\label{def:Dirac}
The \emph{moduli space of framed Dirac monopoles} of weight $\lambda$ is the Gibbons-Manton configuration space $\cT_\lambda$ of Definition \ref{def:Gibbons-Manton-config-space}, which has the total space of the Gibbons-Manton torus bundle \eqref{eq:GM-bundle} as a finite covering.
\end{defn}

\begin{rmk}[\emph{On definitions.}]
\label{rmk:definitions}
Definitions \ref{def:ideal} and \ref{def:Dirac} are not precisely the definitions given in \cite{KottkeSinger2022}. By \cite[Theorem~4.9]{KottkeSinger2022}, the moduli space of ideal monopoles of weight $\lambda$ -- according to their definition -- is equivalent to the space denoted by $\widetilde{\cI}_\lambda$ in Definition \ref{def:ideal}. However, as pointed out in \cite{KottkeSinger2022} (see the Remark on page 53), this is not the correct space to form the boundary hypersurfaces of the compactification $\overline{\cM}_k$ of $\cM_k$, and one should instead pass to the quotient space $\cI_\lambda = \widetilde{\cI}_\lambda / \Sigma_\lambda$. We have therefore made this replacement in Definition \ref{def:ideal}. (The difference between $\cI_\lambda$ and its finite covering space $\widetilde{\cI}_\lambda$ is not significant in \cite{KottkeSinger2022} since they are interested primarily in studying the geometry of these spaces \emph{locally}.)
Similarly, by \cite[Proposition~4.8]{KottkeSinger2022}, the moduli space of framed Dirac monopoles of weight $\lambda$ -- according to their definition -- is equivalent to the total space $\widetilde{\cT}_\lambda$ of the Gibbons-Manton torus bundle \eqref{eq:GM-bundle}. For the same reasons as above, we instead consider the moduli space of framed Dirac monopoles to be the quotient space $\cT_\lambda = \widetilde{\cT}_\lambda / \Sigma_\lambda$ (Definition \ref{def:Dirac}).
Henceforth, we treat Definitions \ref{def:ideal} and \ref{def:Dirac} as the \emph{definitions} of the ideal and framed Dirac monopole moduli spaces respectively.
\end{rmk}

\begin{rmk}
Another small difference between our definition and that of \cite{KottkeSinger2022} concerns the action of the symmetric group $\Sigma_\lambda$. In \cite{KottkeSinger2022}, the ordered centred configuration spaces (cf.~Definition~\ref{defn:ordered-centred}) are defined in a slightly asymmetric way, which does not allow for taking a quotient by $\Sigma_\lambda$ (as we do above), since they single out one point of the configuration to lie at $0 \in \bR^3$. We have modified the definition to be more symmetric by instead requiring the centre of mass to lie at $0$. This does not change the homeomorphism type of the centred ordered configuration space and it has the advantage of having a natural action of the full symmetric group $\Sigma_r$, not just $\Sigma_{r-1}$.
\end{rmk}

\begin{rmk}
When $k=1$, the monopole moduli space $\cM_1^c$, as an $S^1$-space, is simply $S^1$ itself. Thus, according to Definition \ref{def:ideal}, we have $\widetilde{\cI}_{(1,\ldots,1)} = \widetilde{\cT}_{(1,\ldots,1)}$. The moduli space of widely separated magnetic monopoles $\cI_{(1,\ldots,1)}$ (cf.~Terminology \ref{widely-separated}) is therefore the quotient of the total space of the Gibbons-Manton torus bundle $\widetilde{\cT}_{(1,\ldots,1)}$ by the symmetric group $\Sigma_r$.
\end{rmk}

\begin{rmk}[\emph{Higher codimension boundary strata.}]
The space \eqref{eq:partial-compactification} is only a partial compactification of $\cM_k$: it is a manifold with boundary whose interior is $\cM_k$, but it is still non-compact. In a recent preprint \cite{fritzsch2018monopoles}, a full compactification of $\cM_k$ is proposed,\footnote{Although full details of its (recursive) construction are deferred to forthcoming work of the same authors.} which is a smooth manifold with corners that recovers the partial compactification $\overline{\cM}_k$ if one discards corners of codimension greater than $1$. It would be interesting to extend our study of the homology of $\cI_\lambda$ to the deeper boundary strata of this full compactification.
\end{rmk}

\section{Homology stability for configurations with non-local data}
\label{s:stability}

The goal of this section is to prove Proposition \ref{prop:hs-with-non-local-data}, which gives sufficient conditions that imply homology stability for configuration spaces equipped with additional (possibly ``non-local'') parameters.

Labelled configuration spaces, where each separate point of a configuration is equipped with a label taking values in a fixed space, are the most obvious examples of this setting -- we refer to these as configuration spaces with \emph{local} data, since the labels are associated to individual points of the configuration. However, the key observation of this section is that the proof of homology stability requires only weaker properties of the parameters, which are satisfied also in other interesting, \emph{non-local} settings.

In particular, in \S\ref{s:stability-monopoles} we will apply this to our key motivating example of non-local configuration spaces, Gibbons-Manton torus bundles and moduli spaces of ideal monopoles, where the parameters are genuinely non-local, encoding the pairwise interactions of the points of the configuration.

For the general setting of non-local configuration spaces, let us consider a connected manifold $\mbar$ with non-empty boundary and denote its interior by $M$. We first recall the definition of the \emph{stabilisation maps} between the ordered and unordered configuration spaces $F_n(M)$ and $C_n(M)$ (see Definition~\ref{def:configuration-spaces}).

\begin{defn}
\label{def:stabilisation-map}
Choose a collar neighbourhood of $\mbar$, in order words an open neighbourhood $U$ of $\partial\mbar$ and an identification $\varphi \colon U \cong \partial\mbar \times [0,1)$ that restricts to $\varphi(p) = (p,0)$ for $p \in \partial\mbar \subset U$. (This exists by \cite{Brown1962}.) Let $\mhat$ be the result of thickening the collar neighbourhood, i.e.~the union of $\mbar$ and $\partial\mbar \times (-1,1)$ along the identification $\varphi$. Also, choose a diffeomorphism $(-1,1) \cong (0,1)$ that restricts to the identity on $(1-\epsilon,1)$ for some $\epsilon > 0$. Taking the product with the identity on $\partial\mbar$ and extending by the identity on $M \smallsetminus U$, this determines a diffeomorphism $\theta \colon \mhat \cong M$. Finally, choose a basepoint $* \in \partial\mbar$. These choices determine a \emph{stabilisation map}
\begin{equation}
\label{eq:stabilisation-maps-ordered}
F_n(M) \longrightarrow F_{n+1}(M)
\end{equation}
between ordered configuration spaces on $M$ by adjoining the point $(*,-\frac12) \in \mhat$ to a configuration in $M$ and then applying the diffeomorphism $\theta$ to each point, i.e.~the configuration $(p_1,\ldots,p_n)$ is sent to $(\theta(p_1),\ldots,\theta(p_n),\theta((*,-\frac12)))$. This evidently respects the actions of the symmetric groups on $F_n(M)$ and on $F_{n+1}(M)$, so it also descends to a stabilisation map at the level of unordered configuration spaces:
\begin{equation}
\label{eq:stabilisation-maps-unordered}
C_n(M) \longrightarrow C_{n+1}(M),
\end{equation}
as well as intermediate quotients between ordered and unordered configuration spaces, namely:
\begin{equation}
\label{eq:stabilisation-maps-partitioned}
F_n(M)/G \longrightarrow F_{n+1}(M)/H
\end{equation}
for any subgroups $G \subseteq \Sigma_n$ and $H \subseteq \Sigma_{n+1}$ such that the natural inclusion $\Sigma_n \hookrightarrow \Sigma_{n+1}$ takes $G$ into $H$.
\end{defn}

\begin{rmk}
Up to homotopy, the stabilisation maps \eqref{eq:stabilisation-maps-ordered} and \eqref{eq:stabilisation-maps-unordered} depend only on the choice of boundary-component of $\mbar$ containing the basepoint $*$. These maps (or maps homotopic to them) were introduced in \cite[\S 4]{McDuff1975Configurationspacesof} and \cite[Appendix]{Segal1979topologyofspaces}; see also \cite[\S 4]{Randal-Williams2013Homologicalstabilityunordered} or \cite[\S 2.2]{Palmer2018Twistedhomologicalstability}.
\end{rmk}

Let us now consider the sequence
\begin{equation}
\label{eq:classical-stabilisation}
\cdots \to C_n(M) \longrightarrow C_{n+1}(M) \to \cdots
\end{equation}
given by the stabilisation maps \eqref{eq:stabilisation-maps-unordered} and let
\begin{equation}
\label{eq:lifted-stabilisation}
\cdots \to E_n \longrightarrow E_{n+1} \to \cdots
\end{equation}
be another sequence of spaces and maps, equipped with fibrations
\begin{equation}
\label{eq:lifting}
f_n \colon E_n \longrightarrow C_n(M)
\end{equation}
making the evident squares commute. Also choose basepoints $c_n \in C_n(M)$ compatible with the stabilisation maps \eqref{eq:classical-stabilisation}.

\begin{prop}
\label{prop:hs-with-non-local-data}
Fix path-connected spaces $Y$ and $Z$ and suppose that $f_n^{-1}(c_n) = Z^n \times Y$ for all $n$. Fix a basepoint $* \in Z$. Moreover, we assume also that
\begin{itemize}
\item the monodromy $\pi_1(C_n(M)) \to \mathrm{hAut}(Z^n \times Y)$ of \eqref{eq:lifting} is the projection onto the symmetric group followed by the obvious permutation action on the factors of the product $Z^n$;
\item the restriction $Z^n \times Y \to Z^{n+1} \times Y$ of the lifted stabilisation map \eqref{eq:lifted-stabilisation} to fibres over basepoints is the natural inclusion $(z_1,\ldots,z_n,y) \mapsto (*,z_1,\ldots,z_n,y)$.
\end{itemize}
Then the sequence \eqref{eq:lifted-stabilisation} is homologically stable: the map $E_n \to E_{n+1}$ induces isomorphisms on homology in all degrees $\leq n/2 - 1$ with $\bZ$ coefficients and in all degrees $\leq n/2$ with field coefficients.
\end{prop}

\begin{eg}
\label{eg:local-data}
One source of examples of fibrations \eqref{eq:lifting} over configuration spaces $C_n(M)$ equipped with lifted stabilisation maps \eqref{eq:lifted-stabilisation} that satisfy the two conditions of Proposition \ref{prop:hs-with-non-local-data} is \emph{configuration spaces with local data}. This means that we choose a fibration $f \colon E \to \bar{M}$ with path-connected fibres, where $M = \mathrm{int}(\bar{M})$, trivialised over a disc $D \subset \partial\bar{M}$. Then we set
\[
E_n = \bigl\lbrace \{ y_1,\ldots,y_n \} \in C_n(E) \bigm| f(y_i) \neq f(y_j) \text{ for } i\neq j \bigr\rbrace ,
\]
the space of unordered configurations in $M$ where each point $x$ of the configuration is equipped with a label $y \in f^{-1}(x)$. In this setting, the space $Z$ is the fibre of $f$ over $* \in D$. The data in this example is ``local'' in the sense that each label is associated to a single point in the configuration.

However, there also exist labelling data \eqref{eq:lifting} and \eqref{eq:lifted-stabilisation}, satisfying the two conditions of Proposition \ref{prop:hs-with-non-local-data}, that do \emph{not} arise in this way. We will call these ``non-local'' data:
\end{eg}

\begin{defn}
A \emph{system of configuration spaces equipped with non-local data} is a choice of \eqref{eq:lifting} and \eqref{eq:lifted-stabilisation} that do not arise as described in Example \ref{eg:local-data} above.
\end{defn}

\begin{rmk}
Proposition \ref{prop:hs-with-non-local-data}, in the setting of configuration spaces with \emph{local} data, is well-known: see \cite[Appendix A]{KupersMiller2018} or \cite[Appendix B]{CanteroPalmer2015}. The point of this section is to observe that it also holds in a more general setting, requiring just the two assumptions of Proposition \ref{prop:hs-with-non-local-data}, which includes also configuration spaces with \emph{non-local} data. We will see in \S\ref{s:stability-monopoles} that asymptotic monopole moduli spaces are examples of configuration spaces with non-local data: this is our key motivating example. We note, on the other hand, that configuration-mapping spaces, considered in \cite{PalmerTillmann2021}, are in general not examples of configuration spaces with non-local data in the sense of Proposition \ref{prop:hs-with-non-local-data}, as the associated monodromy action does not in general factor through the symmetric group. See \cite[\S 9]{PalmerTillmann2022Pointpushingactions} for a detailed study of the monodromy action for configuration-mapping spaces.
\end{rmk}

In order to prove Proposition \ref{prop:hs-with-non-local-data}, we first need to recapitulate some definitions and results from \cite{Palmer2018Twistedhomologicalstability}. Recall that we are considering a connected manifold $\mbar$ with non-empty boundary whose interior we denote by $M$. Associated to this manifold there is a certain category $\cB(M)$, the \emph{partial braid category on $M$}, whose objects are non-negative integers $\{0,1,2,\ldots\}$ and whose morphisms are ``partial braids'' in $M \times [0,1]$; the precise definition is given in \cite[\S 2.3]{Palmer2018Twistedhomologicalstability}.\footnote{In \cite{Palmer2018Twistedhomologicalstability}, the theory is developed more generally for configuration spaces with (local) labels in a space $X$. We will not need this level of generality here, so we will suppress it (equivalently, we take $X$ to be the one-point space).} This category comes equipped with an endofunctor $S$ that acts by $+1$ on objects as well as a natural transformation $\iota \colon \mathrm{id} \Rightarrow S$. 

\begin{defn}[{\cite[Definitions 2.2 and 3.1]{Palmer2018Twistedhomologicalstability}}]
\label{def:twisted-coefficient-system}
A \emph{twisted coefficient system} for the sequence \eqref{eq:classical-stabilisation} of unordered configuration spaces on $M$, defined over a ring $R$, is a functor $T \colon \cB(M) \to \rmod{R}$. The \emph{degree} of a twisted coefficient system $T$, taking values in $\{-1,0,1,2,3,\ldots\} \cup \{\infty\}$, is defined recursively by setting $\mathrm{deg}(0) = -1$ and declaring that $\mathrm{deg}(T) \leq d$ if and only if $\mathrm{deg}(\Delta T) \leq d-1$, where $\Delta T$ is the cokernel of the natural transformation $T\iota \colon T \Rightarrow TS$.
\end{defn}

\begin{rmk}
In \cite{Palmer2018Twistedhomologicalstability}, the ground ring $R$ is always assumed to be $\bZ$, but everything generalises directly to an arbitrary ground ring $R$.
\end{rmk}

For any twisted coefficient system $T$, the morphisms $\iota_n \colon n \to Sn = n+1$, which between them constitute the natural transformation $\iota$, induce homomorphisms $T(n) \to T(n+1)$. Together with the stabilisation maps \eqref{eq:classical-stabilisation}, these induce homomorphisms
\begin{equation}
\label{eq:map-of-twisted-homology-groups}
H_*(C_n(M);T(n)) \longrightarrow H_*(C_{n+1}(M);T(n+1))
\end{equation}
of twisted homology groups. The main result of \cite{Palmer2018Twistedhomologicalstability} is the following.

\begin{thm}[{\cite[Theorem A]{Palmer2018Twistedhomologicalstability}}]
\label{thm:twisted-homological-stability}
If $T$ is a twisted coefficient system for \eqref{eq:classical-stabilisation} of degree $d$, then the map of twisted homology groups \eqref{eq:map-of-twisted-homology-groups} is an isomorphism in the range of degrees $* \leq \frac12(n-d)$.
\end{thm}

An important family of examples of finite-degree twisted coefficient systems are defined on the category $\mathrm{FI}\sharp$,\footnote{This is denoted $\Sigma$ in \cite{Palmer2018Twistedhomologicalstability}, but we use the more common notation $\mathrm{FI}\sharp$.} which is the category whose objects are non-negative integers and whose morphisms from $m$ to $n$ are the partially-defined injections from $\{1,\ldots,m\}$ to $\{1,\ldots,n\}$. For any manifold $M$, there is a canonical functor $f_M \colon \cB(M) \to \mathrm{FI}\sharp$, so any functor $\mathrm{FI}\sharp \to \rmod{R}$ determines a twisted coefficient system for any manifold $M$.

\begin{construction}[{A generalisation of \cite[Example 4.1]{Palmer2018Twistedhomologicalstability}}]
\label{construction:twisted-coefficient-system}
Choose path-connected spaces $Y,Z$ and a basepoint $* \in Z$. Also choose an integer $q\geq 0$ and a field $K$. There is a functor
\begin{equation}
\label{eq:functor-on-FIsharp}
T_{Z,Y,q,F} \colon \mathrm{FI}\sharp \longrightarrow \rmod{K}
\end{equation}
that acts on objects by $n \mapsto H_q(Z^n \times Y ; K)$ and, on morphisms, sends each partially-defined injection $j \colon \{1,\ldots,m\} \dashrightarrow \{1,\ldots,n\}$ to the map on homology induced by the map $Z^m \times Y \to Z^n \times Y$ defined by $(z_1,\ldots,z_m,y) \mapsto (z_{j^{-1}(1)},\ldots,z_{j^{-1}(n)},y)$. Notice that $j^{-1}(i)$ is either a single element or empty; for the latter case, we interpret $z_\varnothing$ to mean the basepoint $*$ of $Z$.
\end{construction}

\begin{lem}
\label{lem:twisted-coefficient-system}
For any manifold $M$, the twisted coefficient system
\begin{equation}
\label{eq:functor-on-FIsharp-composed-with-fM}
T_{Z,Y,q,F} \circ f_M \colon \cB(M) \longrightarrow \rmod{K}
\end{equation}
given by composing \eqref{eq:functor-on-FIsharp} with the canonical functor $f_M \colon \cB(M) \to \mathrm{FI}\sharp$ has degree at most $q$.
\end{lem}
\begin{proof}
When $Y$ is the one-point space, this is \cite[Lemma 4.2]{Palmer2018Twistedhomologicalstability}. The extra factor of $Y$ in the product does not affect the proof at all (as long as $Y$ is path-connected), so the proof of the general case is identical to that of \cite[Lemma 4.2]{Palmer2018Twistedhomologicalstability}.
\end{proof}

This completes our recapitulation of the necessary definitions and results of \cite{Palmer2018Twistedhomologicalstability}, and we may now complete the proof of Proposition \ref{prop:hs-with-non-local-data}.

\begin{proof}[Proof of Proposition \ref{prop:hs-with-non-local-data}]
We will take field coefficients and prove homological stability up to degree $n/2$. This will automatically imply homological stability up to degree $n/2 - 1$ with integral coefficients (and hence any untwisted coefficients), via the short exact sequences of coefficients
\[
1 \to \bZ/(p^n) \longrightarrow \bZ/(p^{n+1}) \longrightarrow \bZ/(p) \to 1 \qquad\text{and}\qquad 1 \to \bZ \longrightarrow \bQ \longrightarrow \bQ/\bZ \to 1
\]
and the fact that $\bQ/\bZ$ decomposes into the direct sum of $\mathrm{colim}_n(\bZ/(p^n))$ over all primes $p$.

We therefore consider the Serre spectral sequence, with coefficients in a field $K$, associated to the fibration \eqref{eq:lifting} and the map of Serre spectral sequences induced by the stabilisation maps downstairs \eqref{eq:classical-stabilisation} and upstairs \eqref{eq:lifted-stabilisation}. The map of $E^2$ pages is of the form
\begin{equation}
\label{eq:map-of-E2-pages-of-Serre-spectral-sequence}
H_p(C_n(M);H_q(Z^n \times Y ; K)) \longrightarrow H_p(C_{n+1}(M);H_q(Z^{n+1} \times Y ; K)).
\end{equation}
The first assumption of the proposition implies that the local coefficients appearing in the source and target of \eqref{eq:map-of-E2-pages-of-Serre-spectral-sequence} are precisely those arising from the twisted coefficient system \eqref{eq:functor-on-FIsharp-composed-with-fM}. The second assumption implies that the map \eqref{eq:map-of-E2-pages-of-Serre-spectral-sequence} is precisely the one induced by the stabilisation maps \eqref{eq:classical-stabilisation} together with the morphisms $+1 \colon n \to n+1$ of $\mathrm{FI}\sharp$; thus it is the map \eqref{eq:map-of-twisted-homology-groups} for $T = \eqref{eq:functor-on-FIsharp-composed-with-fM}$. By Lemma \ref{lem:twisted-coefficient-system}, this twisted coefficient system has degree at most $q$. Hence Theorem \ref{thm:twisted-homological-stability} implies that \eqref{eq:map-of-E2-pages-of-Serre-spectral-sequence} is an isomorphism for all $p \leq \tfrac12(n - q)$, in particular for all $p+q \leq n/2$. A spectral sequence comparison argument then implies that the map on $H_*(-;K)$ induced by $E_n \to E_{n+1}$ is an isomorphism in degrees $* \leq n/2$.
\end{proof}

\begin{rmk}
One may prove Proposition \ref{prop:hs-with-non-local-data} using the twisted homological stability result \cite[Theorem D]{Krannich2019Homologicalstabilitytopological} instead of the twisted homological stability result \cite[Theorem A]{Palmer2018Twistedhomologicalstability}, although this results in a range of degrees one smaller, namely $n/2 - 1$ for field coefficients and $n/2 - 2$ for integral coefficients.
\end{rmk}

\begin{rmk}
The map \eqref{eq:map-of-E2-pages-of-Serre-spectral-sequence} of $E^2$ pages of Serre spectral sequences is \emph{split-injective in all degrees} by \cite[Theorem A]{Palmer2018Twistedhomologicalstability}. However, this does not in general imply split-injectivity in the limit, so we cannot deduce from this that $E_n \to E_{n+1}$ induces split-injections on homology. Anticipating Remark \ref{rmk:forgetful-maps-not-lifting}, there are obstructions to proving split-injectivity on homology for configurations with non-local data, in contrast to the case of ordinary configurations and twisted homology.
\end{rmk}

\section{Homology stability for asymptotic monopole moduli spaces}
\label{s:stability-monopoles}

Fix a positive integer $c$ and a tuple $\lambda = (k_1,\ldots,k_r)$ of positive integers that sum to $k$. Denote by $\lambda[n]_c$ the tuple $(k_1,\ldots,k_r,c,\ldots,c)$, where there are $n$ appearances of $c$. For simplicity we will assume that $k_i \neq c$ for each $i$ (if this is not the case we may simply remove these entries from $\lambda$ and increase $n$ appropriately). Our main theorem is the following.

\begin{thm}
\label{thm:stability}
There are natural stabilisation maps
\begin{equation}
\label{eq:stab-maps-lifted}
\cT_{\lambda[n]_c} \longrightarrow \cT_{\lambda[n+1]_c} \qquad\text{and}\qquad \cI_{\lambda[n]_c} \longrightarrow \cI_{\lambda[n+1]_c}
\end{equation}
that induce isomorphisms on homology in all degrees $\leq n/2 - 1$ with $\bZ$ coefficients and in all degrees $\leq n/2$ with field coefficients.
\end{thm}

We first prove Theorem \ref{thm:stability} for the Gibbons-Manton configuration spaces $\cT_{\lambda[n]_c}$ in \S\ref{ss:Gibbons-Manton}. We then show in \S\ref{ss:changing-fibre} that homological stability is preserved in general when replacing each circle factor in the torus fibre of $\cT_{\lambda[n]_c}$ with another space that is equipped with a circle action. In particular, we deduce the second part of Theorem \ref{thm:stability}, since moduli spaces of ideal monopoles $\cI_{\lambda[n]_c}$ are special cases of this construction.

\subsection{Gibbons-Manton torus bundles}
\label{ss:Gibbons-Manton}

Recall that the Gibbons-Manton torus bundle $\cT_{\lambda[n]_c}$ has base space $F_{r+n}(\bR^3) / \Sigma_{\lambda[n]_c}$, where $\Sigma_{\lambda[n]_c} = \Sigma_\lambda \times \Sigma_n$. By abuse of notation, we will write
\[
F_{r+n}(\bR^3) / \Sigma_{\lambda[n]_c} =: C_{\lambda,n}(\bR^3).
\]
A point in this space consists of two disjoint configurations in $\bR^3$: one $\lambda$-partitioned configuration of $r$ points and one unordered configuration of $n$ points.

Our first goal in this section is to lift the classical stabilisation maps of configuration spaces
\begin{equation}
\label{eq:stab-maps-classical}
C_{\lambda,n}(\bR^3) \longrightarrow C_{\lambda,n+1}(\bR^3)
\end{equation}
(see Definition \ref{def:stabilisation-map}) to the Gibbons-Manton torus bundles:
\begin{equation}
\label{eq:stab-maps-Gibbons-Manton}
\begin{tikzcd}
\cT_{\lambda[n]_c} \ar[rr,dashed] \ar[d] && \cT_{\lambda[n+1]_c} \ar[d] \\
C_{\lambda,n}(\bR^3) \ar[rr] && C_{\lambda,n+1}(\bR^3).
\end{tikzcd}
\end{equation}

Our second goal is to show that these lifted stabilisation maps satisfy the two hypotheses of Proposition \ref{prop:hs-with-non-local-data}. This will imply homological stability for Gibbons-Manton torus bundles, i.e.~the first part of Theorem~\ref{thm:stability}.

We begin with a lemma about pullbacks of Gibbons-Manton circle factors. To prepare for this, we first choose an explicit concrete model for the stabilisation maps \eqref{eq:stab-maps-classical}; i.e.~we make explicit some of the choices involved in Definition \ref{def:stabilisation-map} in the case $M=\bR^3$. Up to homotopy, this does not make any difference, but it will be convenient for the proof of Lemma~\ref{lem:pullback-GM-circle-factors} below to choose a specific representative of this homotopy class of maps.

\begin{defn}
\label{def:stabilisation-map-R3}
We will in fact replace $\bR^3$ with the open upper half-space $M = \bR^2 \times (0,\infty)$. We may then take $\mbar = \bR^2 \times [0,\infty)$ with the obvious collar neighbourhood, so $\mhat = \bR^2 \times (-1,\infty)$. Take $* = (0,0,0) \in \partial\mbar = \bR^2 \times \{0\}$ as basepoint. With these choices (and identification of $\bR^3$ with $\bR^2 \times (0,\infty)$), the stabilisation map
\begin{equation}
\label{eq:stabilisation-map-Fr}
F_{r-1}(\bR^3) \longrightarrow F_r(\bR^3)
\end{equation}
of Definition \ref{def:stabilisation-map} acts as follows. To a configuration $(x_1,\ldots,x_{r-1})$ in $\bR^2 \times (0,\infty)$, we adjoin the new point $(0,0,-\frac12)$ and then ``push upwards'' the resulting configuration in $\bR^2 \times (-1,\infty)$, i.e., we keep the first two coordinates of all points fixed and modify their third coordinates according to a chosen diffeomorphism $(-1,\infty) \cong (0,\infty)$.
\end{defn}

\begin{lem}
\label{lem:pullback-GM-circle-factors}
Let $\lambda = (k_1,\ldots,k_r)$ for positive integers $k_i$ and write $\lambda' = (k_1,\ldots,k_{r-1})$. Then the pullback of the circle bundle $S_{\lambda,j} \to F_r(\bR^3)$ along the stabilisation map \eqref{eq:stabilisation-map-Fr} is $S_{\lambda',j} \to F_{r-1}(\bR^3)$ if $j\leq r-1$ and a trivial bundle if $j=r$.
\end{lem}
\begin{proof}
Recall that the bundle $S_{\lambda,j} \to F_r(\bR^3)$ is the pullback of the universal $S^1$-bundle on $\mathbb CP^\infty$ along the map $F_r(\bR^3) \to \bC P^\infty$ given by the sum $\sum_{i=1,i\neq j}^r k_i . \iota_{ij}$ where $\iota_{ij} \colon F_r(\bR^3) \to S^2 \subset \bC P^\infty$ is given by
\begin{equation}
\label{eq:iota-ij}
\mathbf{x} = (x_1, \dots, x_r) \longmapsto \frac{x_i-x_j}{\lvert x_i - x_j \rvert}.
\end{equation}
(Recall from Definition \ref{def:stabilisation-map-R3} that we have implicitly replaced $\bR^3$ with $\bR^2 \times (0,\infty)$; the formula above remains true after this replacement.) Its pullback to $F_{r-1}(\bR^3)$ along the stabilisation map \eqref{eq:stabilisation-map-Fr} is therefore given by the same formula, restricting $\iota_{ij}$ to $F_{r-1}(\bR^3)$ along \eqref{eq:stabilisation-map-Fr}.

The key observation is the following. When $i=r$ and we restrict $\iota_{rj}$ to $F_{r-1}(\bR^3)$ along \eqref{eq:stabilisation-map-Fr}, the vertical (third) coordinate of the point $x_r$ will always be smaller than the vertical coordinate of the point $x_j$, due to the choices made in the construction of \eqref{eq:stabilisation-map-Fr} in Definition \ref{def:stabilisation-map-R3}; see Figure~\ref{fig:restriction-of-iota-rj} for a detailed explanation. Thus the right-hand side of \eqref{eq:iota-ij} always takes values in the bottom hemisphere of $S^2 \subset \bC P^\infty$, and hence $\iota_{rj}$ restricted along \eqref{eq:stabilisation-map-Fr} is nullhomotopic. By exactly analogous reasoning, when $j=r$ the map $\iota_{ir}$ restricted along \eqref{eq:stabilisation-map-Fr} takes values in the top hemisphere of $S^2 \subset \bC P^\infty$ and hence is also nullhomotopic.

Putting this all together, we deduce that the map $F_{r-1}(\bR^3) \to \bC P^\infty$ classifying the pullback of $S_{\lambda,r}$ is nullhomotopic, so this pullback is trivial. It also implies that the map $F_{r-1}(\bR^3) \to \bC P^\infty$ classifying the pullback of $S_{\lambda,j}$, for $j\leq r-1$, is the sum $\sum_{i=1,i\neq j}^{r-1} k_i . \iota_{ij}$, which is by definition the map that classifies $S_{\lambda',j}$.
\end{proof}

\begin{figure}
    \centering
    \includegraphics[scale=0.8]{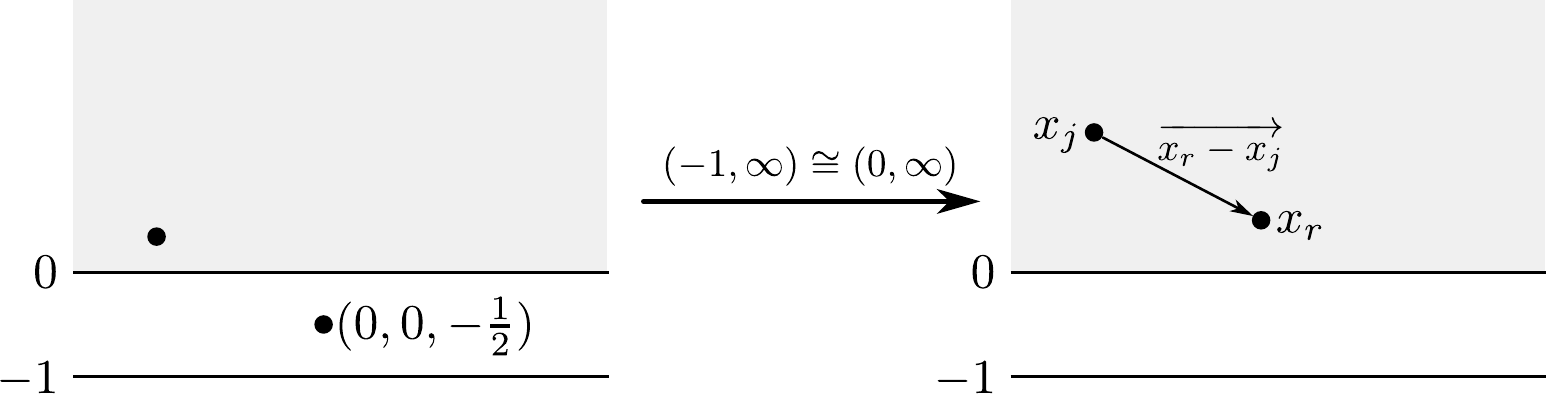}
    \caption{Any configuration in the image of the stabilisation map \eqref{eq:stabilisation-map-Fr} has the form depicted on the right-hand side above (only the points $x_j$ and $x_r$ are actually depicted). Namely, $x_j$ is the image, after applying the chosen diffeomorphism $(-1,\infty) \cong (0,\infty)$ to vertical coordinates, of an arbitrary point in $\bR^2 \times (0,\infty)$, whereas $x_r$ is the image of $(0,0,-\frac12)$. Since the diffeomorphism $(-1,\infty) \cong (0,\infty)$ is order-preserving, the vertical coordinate of $x_j$ is higher than the vertical coordinate of $x_r$. Hence the (normalised) vector from $x_j$ to $x_r$ lies in the bottom hemisphere of $S^2$.}
    \label{fig:restriction-of-iota-rj}
\end{figure}

\begin{rmk}
\label{rmk:action-on-second-cohomology}
Recalling that we denote by $\alpha_{ij}$ the pullback of a fixed generator of $H^2(S^2;\bZ)$ along the map $\iota_{ij} \colon F_r(\bR^3) \to S^2$, the discussion in the proof above implies that the stabilisation map $F_{r-1}(\bR^3) \to F_r(\bR^3)$ acts on $H^2(-;\bZ)$, in the basis \eqref{eq:cohomology-basis}, by $\alpha_{ij} \mapsto \alpha_{ij}$ if $j\leq r-1$ and $\alpha_{ir} \mapsto 0$. It is also easy to see that the automorphism $\sigma_* \colon F_r(\bR^3) \to F_r(\bR^3)$ induced by a permutation $\sigma \in \Sigma_r$ acts on generators of $H^2(F_r(\bR^3);\bZ)$ by $\alpha_{ij} \mapsto \alpha_{\sigma^{-1}(i),\sigma^{-1}(j)}$. It follows from this that the pullback of the circle bundle $S_{\lambda,j}$ along $\sigma_*$ is the circle bundle $S_{\sigma^{-1}(\lambda),\sigma^{-1}(j)}$.
\end{rmk}

\begin{coro}
\label{coro:stab-map-lifts}
The stabilisation map \eqref{eq:stab-maps-classical} lifts to \eqref{eq:stab-maps-Gibbons-Manton}.
\end{coro}
\begin{proof}
Let us write $\mu = \lambda[n+1]_c$ and $\mu' = \lambda[n]_c$. Lemma \ref{lem:pullback-GM-circle-factors} then implies that the pullback of the Gibbons-Manton torus bundle $\widetilde{\cT}_\mu = \bigoplus_{j=1}^{r+n+1} S_{\mu,j} \to F_{r+n+1}(\bR^3)$ along the stabilisation map $F_{r+n}(\bR^3) \to F_{r+n+1}(\bR^3)$ is
\[
\bigoplus_{j=1}^{r+n} S_{\mu',j} \oplus \mathrm{tr} = \widetilde{\cT}_{\mu'} \oplus \mathrm{tr} \longrightarrow F_{r+n}(\bR^3),
\]
where $\mathrm{tr}$ denotes the trivial $S^1$-bundle. We therefore have bundle maps
\begin{equation}
\label{eq:construction-of-bundle-maps}
\begin{tikzcd}
\widetilde{\cT}_{\lambda[n]_c} \ar[rr] \ar[d] && \widetilde{\cT}_{\lambda[n]_c} \oplus \mathrm{tr} \ar[d] \ar[rr] && \widetilde{\cT}_{\lambda[n+1]_c} \ar[d] \\
F_{r+n}(\bR^3) \ar[rr,"\mathrm{id}"] && F_{r+n}(\bR^3) \ar[rr] && F_{r+n+1}(\bR^3),
\end{tikzcd}
\end{equation}
where the left-hand square is an inclusion of a direct summand and the right-hand square is a pullback. This is equivariant with respect to the actions of $\Sigma_\lambda \times \Sigma_n$ and $\Sigma_\lambda \times \Sigma_{n+1}$. Quotienting by these actions, we obtain the lifted stabilisation map \eqref{eq:stab-maps-Gibbons-Manton}.
\end{proof}

In order to apply Proposition \ref{prop:hs-with-non-local-data} to prove the first part of Theorem \ref{thm:stability}, we recall the following general fact about mondromy actions of fibrations.

\begin{lem}
\label{lem:monodromy-of-fibration}
Let $p \colon E \to B$ be a fibration over a based, path-connected space $B$ admitting a universal covering $\pi \colon \widetilde{B} \to B$. Write $\widetilde{p} \colon \widetilde{E} \to \widetilde{B}$ for the pullback of $p$ along $\pi$. Let $F$ denote the fibre of $p$ over the basepoint $b_0 \in B$ and note that the fibre of $\widetilde{p}$ over each point in $\pi^{-1}(b_0) \subset \widetilde{B}$ is also canonically identified with $F$. Then the monodromy action $\pi_1(B) \to \mathrm{hAut}(F)$ of $p$ is equal to
\[
\pi_1(B) \cong \mathrm{Aut}(\pi \colon \widetilde{B} \to B) \longrightarrow \mathrm{hAut}(F),
\]
where the left-hand isomorphism is the action by deck transformations and the right-hand map is given by the action on $\widetilde{E} \to \widetilde{B}$ by pullback.
\end{lem}

\begin{proof}[Proof of Theorem \ref{thm:stability} for $\cT_{\lambda[n]_c}$.]
We first assume that $\lambda = ()$ and $r=0$, so that $\lambda[n]_c$ is the tuple $(c,c,\ldots,c)$ of $n$ copies of $c\geq 1$. We are now in the setting of Proposition \ref{prop:hs-with-non-local-data} with $\eqref{eq:classical-stabilisation} = \eqref{eq:stab-maps-classical}$, $\eqref{eq:lifted-stabilisation} = \eqref{eq:stab-maps-Gibbons-Manton}$, $\eqref{eq:lifting} = \eqref{eq:GM-bundle-2}$ and $Z = S^1$.\footnote{Proposition \ref{prop:hs-with-non-local-data} requires us to fix a basepoint on $Z = S^1$. This may initially appear problematic, since the circle fibres of the Gibbons-Manton circle factors (Definition \ref{def:GM-bundles}) cannot be given consistent basepoints, since the Gibbons-Manton circle factors do not admit global sections. However, Proposition \ref{prop:hs-with-non-local-data} only requires a choice of basepoint on a \emph{single} fibre, namely the fibre over the base configuration, so this issue does not arise.}

To complete the proof under this assumption, it suffices to check the two hypotheses of Proposition \ref{prop:hs-with-non-local-data}. The first hypothesis says that the monodromy $\pi_1(C_n(\bR^3)) \to \mathrm{hAut}(T^n)$ of the Gibbons-Manton torus bundle \eqref{eq:GM-bundle-2} is the obvious permutation action on the circle factors of the torus $T^n$. To check this property, we use Lemma \ref{lem:monodromy-of-fibration}. In our setting, the universal covering of $C_n(\bR^3)$ is $F_n(\bR^3)$ and the pullback of $\cT_{\lambda[n]_c} \to C_n(\bR^3)$ is $\widetilde{\cT}_{\lambda[n]_c} \to F_n(\bR^3)$. The deck transformation action of $\pi_1(C_n(\bR^3)) \cong \Sigma_n$ sends a loop (permutation) $\sigma$ to the obvious automorphism $\sigma_*$ of the ordered configuration space $F_n(\bR^3)$. By Remark \ref{rmk:action-on-second-cohomology}, the action of $\sigma_*$ by pullback on Gibbons-Manton circle factors sends $S_{\lambda[n]_c,j}$ to $S_{\lambda[n]_c,\sigma^{-1}(j)}$ (here we use the fact that $\lambda[n]_c = (c,c,\ldots,c)$, so $\sigma^{-1}(\lambda[n]_c) = \lambda[n]_c$). Hence $\sigma_*$ simply permutes the different circle factors in the Gibbons-Manton torus bundle; in particular its action on the torus fibre simply permutes the different copies of $S^1$, as required.

The second hypothesis of Proposition \ref{prop:hs-with-non-local-data} says that the restriction of the lifted stabilisation map \eqref{eq:stab-maps-Gibbons-Manton} to the fibres over the basepoints is the natural inclusion $T^n \to T^{n+1}$. This is immediate by construction of the lifted stabilisation map: it is given (before quotienting by symmetric groups and therefore also afterwards) by including into a direct sum with a (trivial) circle bundle and then a pullback of bundles.

Proposition \ref{prop:hs-with-non-local-data} therefore implies that the stabilisation map $\cT_{\lambda[n]_c} \to \cT_{\lambda[n+1]_c}$ induces isomorphisms on homology in all degrees $\leq n/2 - 1$ with integral coefficients and in all degrees $\leq n/2$ with field coefficients, under our assumption that $\lambda = ()$.

To complete the proof of Theorem \ref{thm:stability} for $\cT_{\lambda[n]_c}$ we deduce the general case from the special case $\lambda = ()$ that we have just proven. To do this, we first observe that the constructions and results so far generalise directly to Gibbons-Manton torus bundles with \emph{fixed points}. In this setting, we consider the subspace of the configuration space $C_{\lambda,n}(\bR^3)$ where the $\lambda$-partitioned $r$-point configuration $\mathbf{x}$ is fixed and the unordered $n$-point configuration is free to move in the complement of $\mathbf{x}$; in other words, we consider the fibre of the projection $C_{\lambda,n}(\bR^3) \to C_{\lambda}(\bR^3)$ over $\mathbf{x} \in C_{\lambda}(\bR^3)$. Let us denote this subspace by $C_{\lambda,n}(\bR^3 ; \mathbf{x})$ and consider the restriction of $\cT_{\lambda[n]_c} \to C_{\lambda,n}(\bR^3)$ to $C_{\lambda,n}(\bR^3 ; \mathbf{x})$, which we denote by $\cT_{\lambda[n]_c}|_{\mathbf{x}}$. The difference between this setting and the $\lambda = ()$ setting considered above is that (1) the unordered $n$-point configuration now lies in $\bR^3 \smallsetminus \mathbf{x}$, (2) there are $r$ additional Gibbons-Manton circle factors encoding the pairwise interactions of the fixed points $\mathbf{x}$ with the free points and (3) the $n$ Gibbons-Manton circle factors that encode the pairwise interactions of the $n$ free points with each other are now modified to also take into account their interactions with the fixed points $\mathbf{x}$. The arguments above generalise directly to this setting and prove that restricted stabilisation maps
\begin{equation}
\label{eq:restricted-stab-maps}
\cT_{\lambda[n]_c}|_{\mathbf{x}} \longrightarrow \cT_{\lambda[n+1]_c}|_{\mathbf{x}}
\end{equation}
induce isomorphisms on homology in all degrees $\leq n/2 - 1$ with integral coefficients and in all degrees $\leq n/2$ with field coefficients. To deduce the same for the unrestricted stabilisation maps \eqref{eq:stab-maps-Gibbons-Manton}, we note that $\cT_{\lambda[n]_c}|_{\mathbf{x}}$ is the fibre of the composite fibration
\[
\cT_{\lambda[n]_c} \longrightarrow C_{\lambda,n}(\bR^3) \longrightarrow C_\lambda(\bR^3),
\]
where the second map forgets the unordered $n$-point configuration, consider the map of fibrations
\begin{equation}
\label{eq:map-of-fibrations}
\begin{tikzcd}
\cT_{\lambda[n]_c}|_{\mathbf{x}} \ar[rr,"{\eqref{eq:restricted-stab-maps}}"] \ar[d] && \cT_{\lambda[n+1]_c}|_{\mathbf{x}} \ar[d] \\
\cT_{\lambda[n]_c} \ar[rr,"{\eqref{eq:stab-maps-Gibbons-Manton}}"] \ar[dr] && \cT_{\lambda[n+1]_c} \ar[dl] \\
& C_\lambda(\bR^3) &
\end{tikzcd}
\end{equation}
and apply a spectral sequence comparison argument to the corresponding map of Serre spectral sequences.
\end{proof}

\begin{rmk}
\label{rmk:forgetful-maps-not-lifting}
For unordered configuration spaces, the stabilisation maps $C_n(\bR^3) \to C_{n+1}(\bR^3)$ have the additional property that they are split-injective on homology. This is essentially a consequence of the existence of forgetful maps $F_n(\bR^3) \to F_r(\bR^3)$ at the level of ordered configuration spaces that forget the last $n-r$ points of a configuration. Using these maps, standard techniques using transfer maps (see \cite{McDuff1975Configurationspacesof} or \cite{ManthorpeTillmann2014Tubularconfigurationsequivariant}) imply split-injectivity on homology for stabilisation maps of unordered configuration spaces. We record here the observation that the forgetful maps
\begin{equation}
\label{eq:forgetful-maps}
\tau_{n,r} \colon F_n(\bR^3) \longrightarrow F_r(\bR^3)
\end{equation}
do \emph{not} naturally lift to Gibbons-Manton torus bundles (in contrast to the stabilisation maps, which do lift, by Corollary \ref{coro:stab-map-lifts}). In order to lift $\tau_{n,r}$ to Gibbons-Manton torus bundles $\widetilde{\cT}_{\lambda} \to \widetilde{\cT}_{\lambda|_r}$, where $\lambda = (k_1,\ldots,k_n)$ and $\lambda|_r = (k_1,\ldots,k_r)$, one would like it to be true that the pullback of the circle bundle $S_{\lambda|_r,j}$ along $\tau_{n,r}$ is $S_{\lambda,j}$ --- given this, one would then be able to pre-compose the pullback of $\widetilde{\cT}_{\lambda|_r}$ with the projection of $\widetilde{\cT}_\lambda$ onto a sub-direct-sum. However, this is false. For every $i<j\leq r$, the pullback of the cohomology class $\alpha_{ij}$ along $\tau_{n,r}$ is $\alpha_{ij}$, so we have
\[
\tau_{n,r}^* \left( \sum_{\substack{i=1 \\ i\neq j}}^{r} k_i . \alpha_{ij} \right) = \sum_{\substack{i=1 \\ i\neq j}}^{r} k_i . \alpha_{ij}.
\]
The left-hand side classifies the pullback of $S_{\lambda|_r,j}$ along $\tau_{n,r}$, but the right-hand side classifies $S_{\lambda,j}$ only if $k_{r+1} = \cdots = k_n = 0$, which is impossible since all $k_i$ are assumed positive.

More informally, one could say that the reason why we cannot naturally lift forgetful maps to Gibbons-Manton torus bundles is because of the \emph{non-local} nature of the additional circle parameters: each circle parameter is associated to \emph{all} configurations points simultaneously, since it encodes the pairwise interactions of one of the points with all of the others. Thus there is no well-defined way of forgetting a subset of the configuration points in the presence of these non-local parameters.
\end{rmk}

\subsection{Changing the fibre}
\label{ss:changing-fibre}

For a sequence of spaces $\mathbf{Z} = \{Z_1,Z_2,\ldots\}$, we will consider the family of finite products of the form $Z_\lambda = Z_{k_1} \times \cdots \times Z_{k_r}$ for tuples $\lambda = (k_1,\ldots,k_r)$ of positive integers. If each $Z_i$ is a $G$-space for some topological group $G$, we consider each $Z_\lambda$ as a $G$-space via the diagonal action.

\begin{defn}
Let $\mathbf{Z}$ be a sequence of $S^1$-spaces and let $\lambda = (k_1,\ldots,k_r)$. Let $\widetilde{\cT}_\lambda(\mathbf{Z})$ be the total space of the fibre bundle obtained from the principal $T^r$-bundle $\widetilde{\cT}_\lambda$ by the Borel construction:
\[
\widetilde{\cT}_\lambda(\mathbf{Z}) = \widetilde{\cT}_\lambda \times_{T^r} Z_\lambda \longrightarrow F_r(\bR^3).
\]
We then let $\cT_\lambda(\mathbf{Z}) = \widetilde{\cT}_\lambda(\mathbf{Z}) / \Sigma_\lambda$, where $\Sigma_\lambda$ acts diagonally on $\widetilde{\cT}_\lambda$ and on the finite product $Z_\lambda$. The \emph{Gibbons-Manton $\mathbf{Z}$-bundle} of weight $\lambda$ is the space $\cT_\lambda(\mathbf{Z})$. It is the total space of a fibre bundle
\begin{equation}
\label{eq:GM-Zstar-bundle}
\cT_\lambda(\mathbf{Z}) \longrightarrow F_r(\bR^3) / \Sigma_\lambda
\end{equation}
with fibre $Z_\lambda$.
\end{defn}

In particular, we have $\cI_\lambda = \cT_\lambda(\mathbf{Z})$ for $\mathbf{Z} = \{ \cM_1,\cM_2,\cM_3,\ldots \}$. We now prove:

\begin{thm}
\label{thm:GM-Zstar-stability}
For any sequence $\mathbf{Z} = \{Z_1,Z_2,\ldots\}$ of path-connected $S^1$-spaces, there are natural stabilisation maps
\begin{equation}
\cT_{\lambda[n]_c}(\mathbf{Z}) \longrightarrow \cT_{\lambda[n+1]_c}(\mathbf{Z})
\end{equation}
that induce isomorphisms on homology in all degrees $\leq n/2 - 1$ with $\bZ$ coefficients and in all degrees $\leq n/2$ with field coefficients.
\end{thm}

Theorem \ref{thm:stability} corresponds to two special cases of Theorem \ref{thm:GM-Zstar-stability}, namely the sequences $\{S^1,S^1,\ldots\}$ and $\{\cM_1,\cM_2,\ldots\}$ of $S^1$-spaces. It therefore remains only to prove Theorem \ref{thm:GM-Zstar-stability}.

\begin{proof}[Proof of Theorem \ref{thm:GM-Zstar-stability}]
The proof is a direct generalisation of the proof of Theorem \ref{thm:stability} for $\cT_{\lambda[n]_c}$, so we just explain the differences. First of all, the lifts of the stabilisation maps exist by the proof of Corollary \ref{coro:stab-map-lifts}, where we additionally apply the (functorial) Borel construction to the outer square of \eqref{eq:construction-of-bundle-maps} before quotienting by the symmetric group actions.

We begin by assuming that $\lambda = ()$ and $r=0$, so that $\lambda[n]_c = (c,c,\ldots,c)$ where there are $n$ copies of $c\geq 1$. We are therefore in the setting of Proposition \ref{prop:hs-with-non-local-data} with $Z = Z_c$. The two hypotheses of that proposition are satisfied by the same argument as in the proof of Theorem \ref{thm:stability} for $\cT_{\lambda[n]_c}$, together with the evident observation that applying the Borel construction that replaces each circle factor in the fibre with the $S^1$-space $Z_c$ has the effect, on fibres, that permutation maps $T^n \to T^n$ and natural inclusions $T^n \to T^{n+1}$ are sent to the corresponding permutation maps $(Z_c)^n \to (Z_c)^n$ and natural inclusions $(Z_c)^n \to (Z_c)^{n+1}$. Thus Proposition \ref{prop:hs-with-non-local-data} completes the proof in the case $\lambda = ()$.

This generalises to Gibbons-Manton $\mathbf{Z}$-bundles with \emph{fixed points} exactly as for Gibbons-Manton torus bundles with fixed points, and one may then deduce the general case of the theorem from this by a spectral sequence comparison argument applied to the analogue of the diagram \eqref{eq:map-of-fibrations}.
\end{proof}

\enlargethispage{20pt}

\ack{The first author is grateful to Michael Singer for introducing him to asymptotic monopole moduli spaces, and for asking the question of whether they are homologically stable. The first author was partially supported by a grant of the Romanian Ministry of Education and Research, CNCS - UEFISCDI, project number PN-III-P4-ID-PCE-2020-2798, within PNCDI III. The authors are grateful to the anonymous referee for their careful reading of and helpful comments on an earlier version of this paper.}


\vskip2pc

\bibliographystyle{RS}

\bibliography{biblio}

\end{document}